\documentclass{ws-ijbc}
\usepackage{ws-rotating}     
\usepackage{graphicx}
\usepackage{epstopdf}
\begin{document}


\markboth{J. Kov\'a\v c, K. Jankov\'a}{Random dynamical systems generated by two Allee maps}

\title{Random dynamical systems generated by two Allee maps}

\author{Jozef Kov\'a\v c 
}

\address{Department of Applied Mathematics and Statistics, Faculty of Mathematics, Physics and Informatics, Comenius University in Bratislava, Mlynsk\'a dolina\\
Bratislava, 84248, Slovakia\\
jozef.kovac@fmph.uniba.sk
}

\author{Katar\'ina Jankov\'a}
\address{Department of Applied Mathematics and Statistics, Faculty of Mathematics, Physics and Informatics, Comenius University in Bratislava, Mlynsk\'a dolina\\
Bratislava, 84248, Slovakia\\
jankova@fmph.uniba.sk
}

\maketitle


\begin{abstract}
In this paper, we study random dynamical systems generated by two Allee maps. Two models are considered - with and without small random perturbations. It is shown that the behavior of the systems is very similar to the behavior of the deterministic system if we use strictly increasing Allee maps. However, in the case of unimodal Allee maps, the behavior can dramatically change irrespective of the initial conditions. 
\end{abstract}

\section{Introduction}

\noindent The Allee effect, which was first described by W.C. Allee [1932], is a biological phenomenon characterized by a correlation between the population size and the per capita population growth rate.  In the literature, the most frequently mentioned reason of this effect is the difficulty of finding mates in a smaller population (\cite{BB}). There are several possible scenarios for the behavior of the population size according to the strength of the Allee effect (see \cite{BB}). 

The Allee effect was investigated from various points of view using models of dynamical and semi-dynamical systems. Interesting results were obtained for non-autonomous  periodic systems. We mention several recent achievements. A special attention was given to the study of the Beverton-Holt equation. In connection with a periodic model based on this equation, two conjectures were formulated in \cite{Cushing}. These conjectures were proved in the later results of Elaydi and Sacker [2005] and also for more general systems in \cite{Kon} (see also \cite{Elaydi1}). It was shown that periodic fluctuations in the environment have, in a certain sense, a deleterious effect on the average population in the corresponding models. Periodic fluctuations in a context of economical models with an Allee effect were studied in \cite{Luis2}. A class of unimodal Allee maps was introduced in \cite{Luis} where properties and stability of the three fixed points in non-autonomous periodic dynamical systems with period~2 were studied. 
\par  Some recent works take into account also randomly varying systems with the Allee effect. The stochastic stability of such systems was investigated in \cite{Haskell} and \cite{Bezandry}. Environmental stochasticity in a connection with Allee effect was examined and persistence, asymptotic extinction and conditional persistence for stochastic difference equations was analyzed in \cite{Roth}.

In this paper, only the extinction-survival scenario is studied. If the population size drops below a certain critical value, the population becomes extinct in the long run. If the size exceeds the critical value, the population is generally able to survive. Therefore the function (called an Allee map) $f:[0,b]\to [0,b]\ (b\le\infty)$ describing population dynamics with this effect 
\begin{equation}
x_{n+1}=f(x_n)\label{eq1}
\end{equation}
($x_n$ is the population size in the $n$-th generation) must satisfy the following conditions: 
\begin{itemlist}[(3)]
\item $f$ has three fixed points - zero, which is a stable fixed point, an unstable threshold point $A_f>0$ and a fixed point $K_f>A_f$, which can (but does not have to) be stable,
\item $f(x)<x$ for all $x \in (0,A_f) \cup (K_f,b]$,
\item $f(x)>x$ for all $x \in (A_f,K_f)$.
\end{itemlist}

We consider a pair of continuous Allee maps $f, g$ and two random dynamical systems generated by these maps. The first system is given by
\begin{equation}
X_{n+1}= 
\begin{cases} 
f(X_n) & \text{with probability } p,  \\ 
g(X_n) & \text{with probability } 1-p
\end{cases}\label{eq2}
\end{equation}
while the other one, 
\begin{equation}
Y_{n+1}= 
\begin{cases} 
\chi(f(Y_n)+\varepsilon_n) & \text{with probability } p,  \\ 
\chi(g(Y_n)+\varepsilon_n) & \text{with probability } 1-p, 
\end{cases}\label{eq3}
\end{equation}
contains also perturbations. In~(\ref{eq3}), $\{\varepsilon_n\}_{n=0}^\infty$ are continuous i.i.d. random variables with a positive density on the interval 
$(-\delta,\delta)$, $\delta>0$ and $\chi:R\to R$ is a function defined as follows: 
\begin{equation}
\chi(x)= 
\begin{cases} 
	0 & \text{if } x< 0, \\
	x & \text{if } x\in [0,b], \\
	b & \text{if } x>b.
\end{cases}\label{eq4}
\end{equation}
More formally (see \cite{Bhat}), we can write $X_{n+1}=H_n(X_n)$, where $\{H_n\}_{n=1}^\infty$ are i.i.d. random functions such that $P(H_n=f)=p$ and $P(H_n=g)=1-p$ for some $p\in (0,1)$ and $X_0=x_0\in [0,b]$, similarly for model~(\ref{eq3}).
Given $X_n$, the variable $X_{n+1}$ does not depend on $X_{n-1}, X_{n-2}, \ldots$, and thus the process $\{X_n\}_{n=0}^\infty$ 
is a  Markov process with $X_0=x_0\in [0,b]$.  The same is true for the process $\{Y_n\}_{n=0}^\infty$ with $Y_0=y_0\in [0,b]$.

Braverman [2011] considered a model similar to~(\ref{eq3}), i.e.,  with one function and scarce random perturbations.
It was shown that such perturbations can, under some conditions, cause the extinction or survival of the population regardless of its initial size. 

Section~\ref{sec2} will focus on the behavior of models (\ref{eq2}) and (\ref{eq3}) for strictly increasing functions $f$ and $g$. Unimodal functions $f$ and $g$ will be analyzed in Section~\ref{Sec3} - in both models, the population becomes extinct under some conditions, irrespective of its initial size. We can thus observe a behavior similar to the Parrondo's paradox \cite{Parrondo}.

\section{Strictly increasing Allee maps}
\label{sec2}

In this section, we assume that the functions $f$ and $g$ are continuous and strictly increasing. An example of such Allee map,
\begin{equation}
f(x)=\frac{\rho x^2}{A+x^2}, \ \rho>2\sqrt{A},\ \rho>0, A>0 \label{eq5}
\end{equation}
can be found in \cite{Hoppen}. It is also a special case of a new model presented in \cite{Elaydi3}. 
\begin{figure}[h]
\begin{center}
\includegraphics[height=8cm,width=8cm]{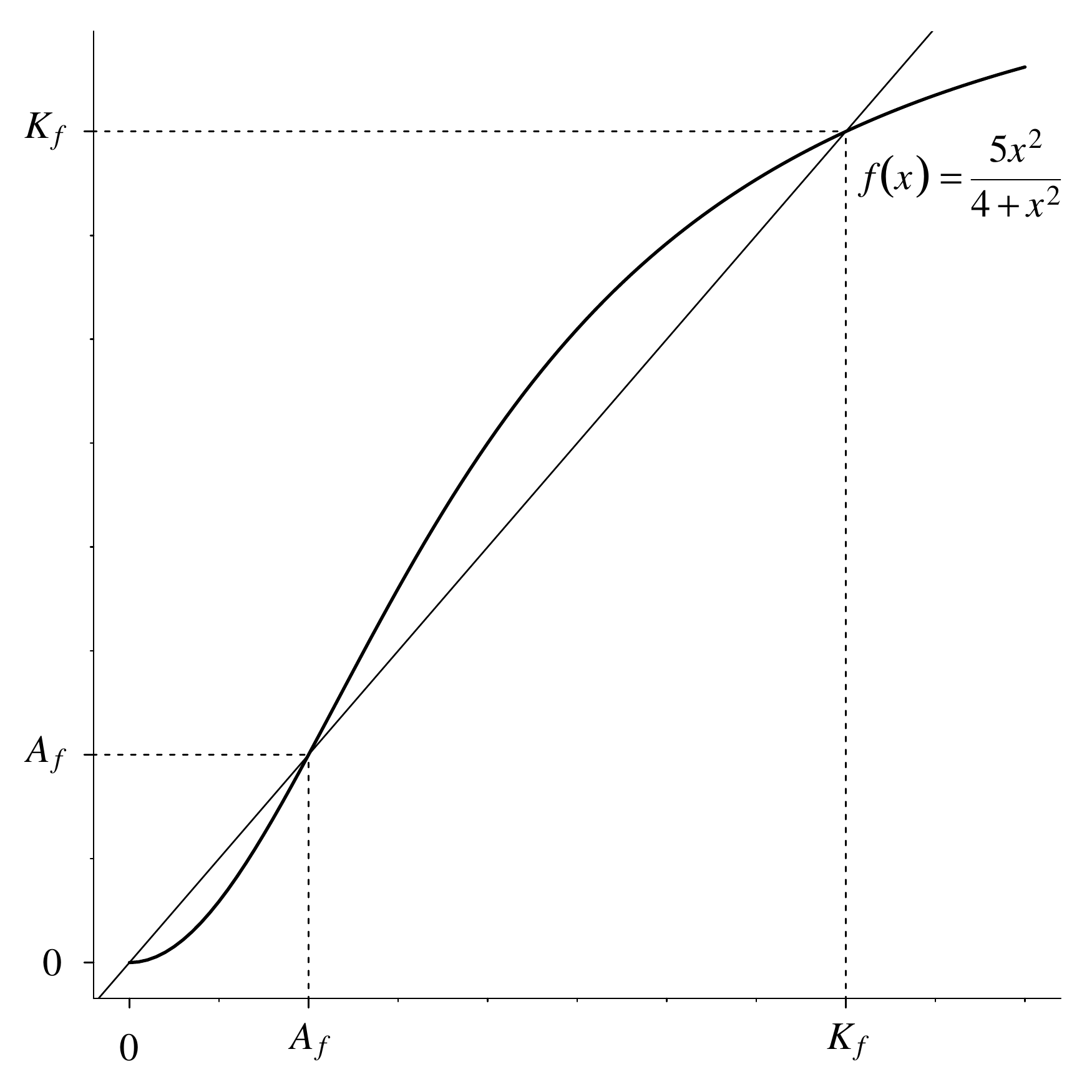}
\end{center}
\caption{Example of a strictly increasing Allee map.}
\label{fig1}
\end{figure}
The behavior of the non-stochastic model (i.e., model~(\ref{eq1})) is very simple. If $x_0>A_f$, then $\lim\limits_{n\to\infty}x_n=K_f$ and if $x_0<A_f$, then $\lim\limits_{n\to\infty}x_n=0$. Before we formulate theorems about models (\ref{eq2}) and (\ref{eq3}), it is useful to mention the following lemma and its corollary.
\begin{lemma}
\label{lem1}
Let $\{X_n\}_{n=0}^\infty$ be a Markov process and let $X_n\in B$ for any $n\in N$. Assume that $B_1\subset B$ is such that $P(X_{n+k}\in B_1\mid X_n\in B)\ge \lambda$ for some $\lambda>0$ and $k\in N$. Then $P(\exists n_0\in N: X_{n_0}\in B_1)=1$.
\end{lemma}
A proof of this lemma can be found in \cite{Jankova}.
\begin{corollary}
\label{cor1}
Let $I_1\subset I_2$ be intervals such that the following conditions are satisfied:
\begin{arabiclist}[(2)]
\item $f(I_i)\subset I_i$ and $g(I_i)\subset I_i$, $i=1,2$.
\item For any $x \in I_2$ $f^m(x)\in I_1$ or $g^m(x)\in I_1$ for some $m \in N$. 
\end{arabiclist}
Then $P(\exists n_0 \in N:X_n\in I_1 \ \forall n \ge n_0 \mid x_0\in I_2)=1$.
\end{corollary}
\begin{theorem}
\label{th1}
Let $A_f<A_g<K_f<K_g$. Then 
\begin{romanlist}[(iii)]
\item if $x_0\le A_f$, then $p_0\equiv P(\lim_{n\to \infty} X_n=0)=1$,
\item if $x_0\ge A_g$, then $p_1\equiv P(\exists n_0 \in N: X_n\in [K_f,K_g]\ \forall n\ge n_0)=1$,
\item if $x_0\in (A_f,A_g)$, then $p_0>0$, $p_1>0$ and $p_0+p_1=1$.
\end{romanlist}
\end{theorem}

\begin{proof}
\begin{romanlist}[(iii)]
\item Since $\lim_{n\to\infty} g^n(A_f)=0$, for any $\eta>0$ there exists an $m\in N$ such that $g^m(A_f)<\eta$. However, the function $g$ is strictly increasing, so $g^m$ is also strictly increasing, hence for any $x\le A_f$ we have $g^m(x)<\eta$. It follows that we can apply Corollary 1 to intervals $I_1=[0,\eta)$ and $I_2=[0,A_f]$. This can be done for an arbitrary small $\eta>0$, hence if $x_0\le A_f$, then $p_0=1$.
\item If $x_0 \in [A_g,K_g]$, then $X_n\in [A_g,K_g]$ for all $n \in N$. Let us denote $B_1$ the set $[K_f,K_g]$. If $x_0 \in B_1$ then $K_f=f(K_f)\le f(x)\le x\le K_g$ and also $K_f\le x\le g(x)\le g(K_g)=K_g$.  Hence $f(B_1)\subset B_1$ and $g(B_1)\subset B_1(*)$.\\
Next, $A_g<f(A_g)<K_f$, so $g^n(f(A_g))\nearrow K_g$ and hence for some $m\in N$ $g^m(f(K_g))\in [K_f,K_g]$. Moreover, function $g^m\circ f$ is increasing, so it follows that $g^m(f(x))\in [K_f,K_g]$ for all $x\in [A_g,K_g]$. \\
Combining $(*)$ and Lemma 1, applied to the sets $B=[A_g,K_g]$ and $B_1$ and $\lambda=p(1-p)^m$, we get $p_ 1=1$. \\
If $x_0\ge K_g$, then $f^n(x_0)\searrow K_f$ and we can apply Corollary 1 to the intervals $I_1=[K_f,K_g]$ and $I_2=[K_f,x_0]$.
\item If $x_0\in (A_f,A_g)$, then $f^m(x_0)>A_f$ and $g^m(x_0)<A_f$ for some $m\in N$. Hence $p_0>0$ and $p_1>0$. \\
It remains to show that $P(X_n\in (A_f,A_g) \ \forall n\in N)=0$.
Let $C=\frac{A_f+A_g}{2}$. Again, for some $m\in N$ we have $f^m(C)>A_g$ and since $f$ is increasing, $f^m(x)>A_g$ for any $x\ge C$. Similarly $g^m(x)<A_f$ for all $x\le C$. Hence we can apply Lemma 1 to the sets $[0,b]$ and $[0,A_f]\cup[A_g,b]$ with $\lambda=p^m(1-p)^m$.
\end{romanlist}
\end{proof}
\begin{theorem}
\label{th2}
Let $A_f<A_g<K_f<K_g$. Take a $\delta>0$ such that the sets
\begin{itemlist}[(3)]
\item $U_1=\{x\in (0,A_f): \min(x-f(x), x-g(x))\ge \delta\},$
\item $U_2=\{x\in (A_g,K_f): \min(f(x)-x, g(x)-x)\ge \delta\},$
\item $U_3=\{x\in (K_g,b): \min(x-f(x), x-g(x))\ge \delta\}$
\end{itemlist}
are nonempty. Assume that for every $x\in [0,b]$ there exists an $x^*\ge x$ such that $\min(x^*-f(x^*), x^*-g(x^*))\ge \delta$.
Then 
\begin{romanlist}[(iii)]
\item if $y_0\in [0,z_1]$, then $p_0\equiv P(\exists n_0 \in N: Y_n\in [0,w_1) \ \forall n\ge n_0)=1$,
\item if $y_0\in [w_2,b]$, then $p_1\equiv P(\exists n_0 \in N: Y_n\in (z_2,w_3) \ \forall n\ge n_0)=1$,
\item if $y_0\in (z_1,w_2)$, then $p_0>0$, $p_1>0$ and $p_0+p_1=1$, 
\end{romanlist}
where $w_i=\inf U_i,\ i=1,2,3$ and $z_j=\sup U_j,\ j=1,2$ (see Fig.~\ref{fig2}).

\end{theorem}
\begin{figure}[h]
\begin{center}
\includegraphics[height=8cm,width=8cm]{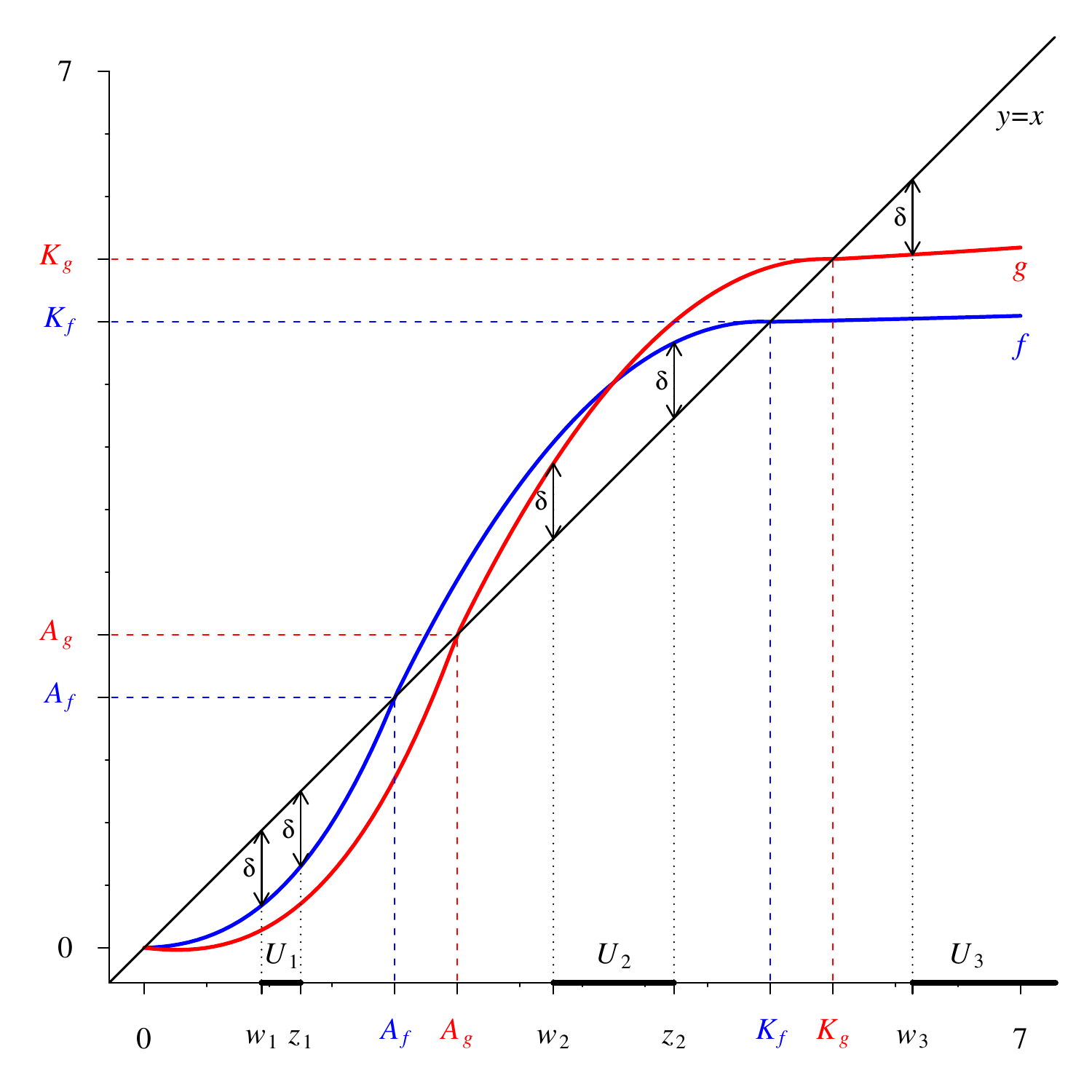}
\end{center}
\caption{Example of functions satisfying the conditions of Theorem \ref{th2}.}
\label{fig2}
\end{figure}
\begin{proof}
\begin{romanlist}[(iii)]
\item Let $d_1= \min(w_1-f(w_1),w_1-g(w_1))$. If $Y_n\in [0,w_1)$, then we have
\begin{equation}
f(Y_n)+\varepsilon_n< f(w_1)+\varepsilon_n< f(w_1)+\delta \le f(w_1)+d_1\le f(w_1)+w_1-f(w_1)=w_1.
\end{equation}
Similar inequalities hold also for the function $g$, hence $Y_{n+1}\in [0,w_1)$. Similarly, if $Y_n\in [0,z_1]$, then $Y_{n+1}\in [0,z_1]$. Moreover, $f^n(z_1)\searrow 0$, therefore we can apply Lemma 1 to the sets $B=[0,z_1]$ and $B_1=[0,w_1)$ with $\lambda=(pP(\varepsilon_n<0))^k$ for some $k\in N$.
\item Let $d_2=\min(f(z_2)-z_2,g(z_2)-z_2)$ and $d_3=\min(w_3-f(w_3),w_3-g(w_3))$. If $Y_n \in (z_2,w_3)$, then $Y_{n+1} \in (z_2,w_3)$, because 
\begin{equation}
f(Y_n)+\varepsilon_n>f(z_2)+\varepsilon_n>f(z_2)-\delta \ge f(z_2)-d_2\ge f(z_2)-(f(z_2)-z_2)=z_2
\end{equation} and
\begin{equation}
f(Y_n)+\varepsilon_n<f(w_3)+\varepsilon_n<f(w_3)+\delta \le f(w_3)+d_3\le f(w_3)+w_3-f(w_3)=w_3,
\end{equation}
similarly for the function $g$. If $Y_n>w_3$, then by the assumption there exists a $z_3>Y_n$ such that $\min(z_3-f(z_3),z_3-g(z_3))\ge \delta$. Again, it can be shown that if $Y_n\in (z_2,z_3)$, then $Y_{n+1}\in (z_2,z_3)$, and so we can apply \mbox{Lemma 1} (like in the case (i)) to the sets $B=(z_2,z_3)$ and $B_1=(z_2,w_3)$, because $f^n(z_3)\searrow K_f \in (z_2,w_3)$. The case $Y_n>w_2$ can be shown similarly. 
\item If $z_1<Y_n\le A_f$, then the probability that $Y_{n+k}\le z_1$ for some $k\in N$ is obviously non-zero. However, we will show in the following that also $Y_{n+m}>A_f$ with non-zero probability for some $m\in N$. \\
Let us denote $h$ the function defined by $h(x)=\min(x-f(x),x-g(x))$. Since $f$ and $g$ are continuous functions, $h$ is also continuous and it follows that it attains its maximum on the interval $[Y_n,A_f]$ (denote this maximum $\mu$). Next, let $\eta=\frac{\delta-\mu}{2}$. Since $z_1=\sup U_1$ and $\operatorname{argmax} \limits_{x\in[Y_n,A_f]} h(x)>z_1$, it follows that $\mu<\delta$, and so $0<\eta<\delta$. Hence the probability that $\varepsilon_n$ takes a value from $(\delta-\eta,\delta)$ is non-zero (because the density of $\varepsilon_n$ is positive on $(-\delta,\delta)$).\\
Without loss of generality, assume that $Y_n-f(Y_n)\le \mu$, and hence $Y_n\le f(Y_n)+\mu$ (if it is not true for the function $f$, then it is true for $g$). Hence with non-zero probability we have
\begin{equation}
Y_{n+1}=f(Y_n)+\varepsilon_n>f(Y_n)+\delta-\eta = f(Y_n)+\mu+\eta \ge Y_n+\eta
\end{equation}  
(the second equality is obtained by the definition of $\eta$). Similarly, we can continue with $Y_{n+1}, Y_{n+2}$, etc. and with the same $\eta$. Finally, if an $m\in N$ is sufficiently large,  $Y_{n+m}>Y_n+m\eta >A_f$ with a positive probability. But if $Y_{n+m}>A_f$, then obviously $Y_{n+m+l}>w_2$ with a non-zero probability for some $l \in N$.\\
Similarly, it can be shown that if $A_g\le Y_n<w_2$, then with a non-zero probability $Y_{n+k}>w_2$ for some $k\in N$ and also $Y_{n+m}<z_1$ for some $m\in N$. The case $Y_n \in (A_f,A_g)$ and the fact that $P(Y_n\in (z_1,w_2) \ \forall n \in N)=0$ can be shown as in Theorem \ref{th1}.
\end{romanlist}
\end{proof}
For $A_f<A_g<K_g<K_f$ or for the cases where $A_f=A_g$ or $K_f=K_g$,  very similar statements hold.

\section{Unimodal Allee maps}
\label{Sec3}

In this section, we assume that there exists a $B_f\in (A_f,K_f)$ such that the function $f$ is strictly increasing on $(0,B_f)$ and strictly decreasing on $(B_f,b)$ (denote $M_f=f(B_f)$ the maximum of the function $f$). We also assume that the function $f$ is continuous. An example of such Allee map is the function
\begin{equation}
\label{eq6}
f(x)=\frac{Gbx}{(x-T)^2+b},\ b,T>0,\ G>1, 
\end{equation}
see \cite{Asmussen}.
\begin{figure}[h]
\begin{center}
\includegraphics[height=8cm,width=8cm]{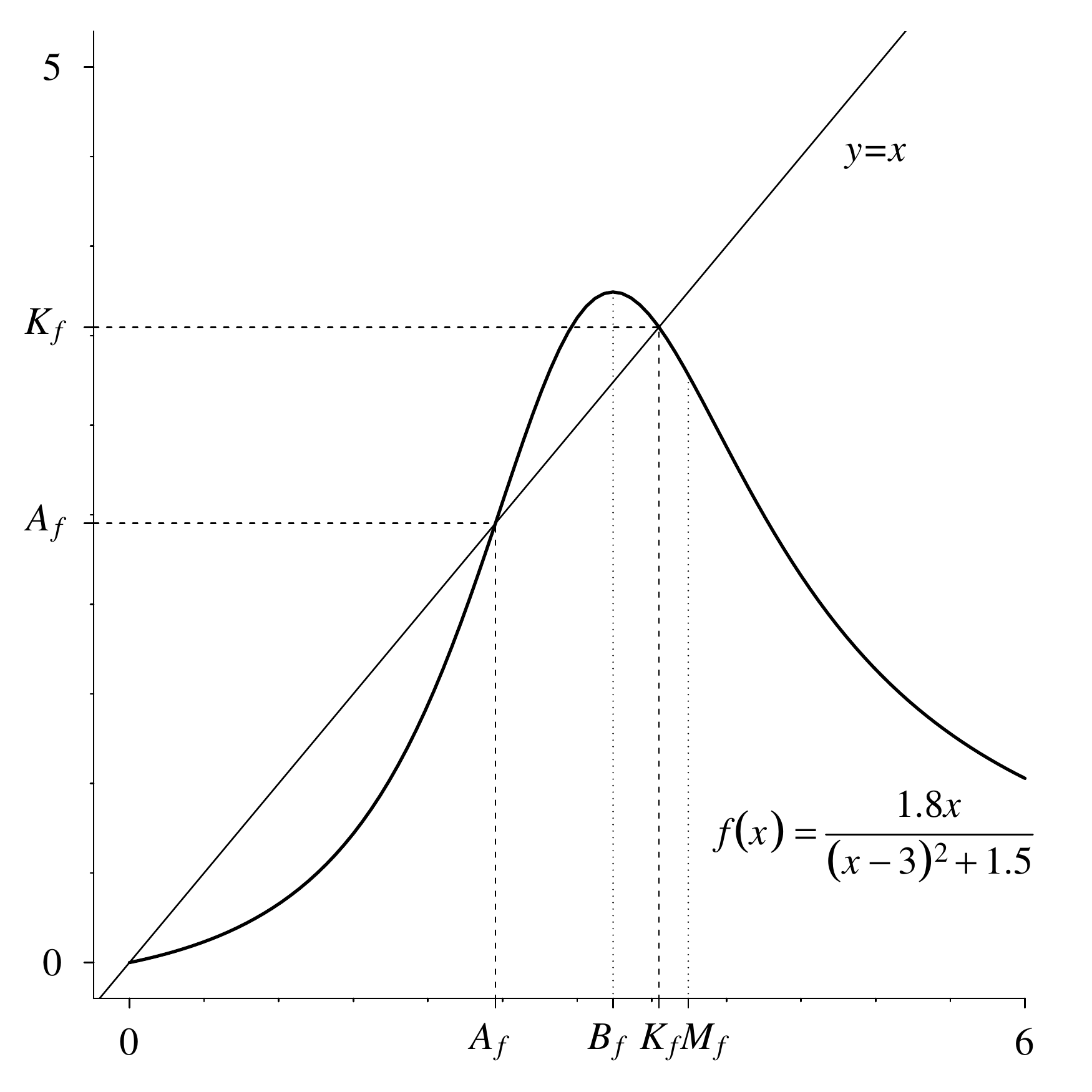}
\end{center}
\caption{Example of a unimodal Allee map.}
\label{fig3}
\end{figure}
Similarly, let $g$ be a continuous unimodal function with the maximum $M_g$ attained at the point $B_g$. Let $M=\max (M_f,M_g)$. For every $x\in [0,b]$ it holds that $\max (f(x),g(x))\le M$. Thus in this section we assume that $b=M$. Again, if $x_0<A_f$ in the deterministic model~(\ref{eq1}), then $\lim_{n\to \infty}x_n=0$, but if $x_0>A_f$, the behavior of the model can vary according to the function $f$. However, if $x_0>A_f$, a sufficient condition for the survival of the population  is $f(M_f)>A_f$ (in this case the population size $x_n$ never drops below the threshold value $A_f$).

As in Section~\ref{sec2}, we first need the following corollary of Lemma~\ref{lem1}.
\begin{corollary}
\label{cor2}
Let the sets $E$ and $F$ be  such that there exists $m \in N$ with the property 
\begin{equation}
P(X_{n+m}\in F|X_n\in E)\ge \lambda
\end{equation}
for some $\lambda>0$. Then 
\begin{equation*}
P(\{n \in N: X_n \in F\}\text{ is unbounded}\mid\{n \in N: X_n \in E\} \text{ is unbounded})=1.
\end{equation*}
(If the process returns infinitely many times to the set $E$, then it almost surely returns infinitely many times to the set $F$).
\end{corollary}

\begin{theorem}
\label{th3}
Let $A_f<A_g$, $f$ be differentiable and $|f^{'}(x)|<1$ for every $x\in(B_f,M_f)$. Assume that there exist  functions $h_1, h_2, \ldots h_m \in \{f,g\}$ for which $h_1\circ h_2\circ \ldots \circ h_m (K_f) < A_f.$ Then $P(\lim_{n \to \infty} X_n=0)=1$ for every $x_0$.
\end{theorem}
\textbf{Example.} Consider two functions $f(x)=\frac{2.2x}{(x-3)^2+2}$ and $g(x)=\frac{1.3x}{(x-3.3)^2+1}$ (see Fig~\ref{fig4}).
\begin{figure}[h]
\begin{center}
\includegraphics[height=8cm,width=8cm]{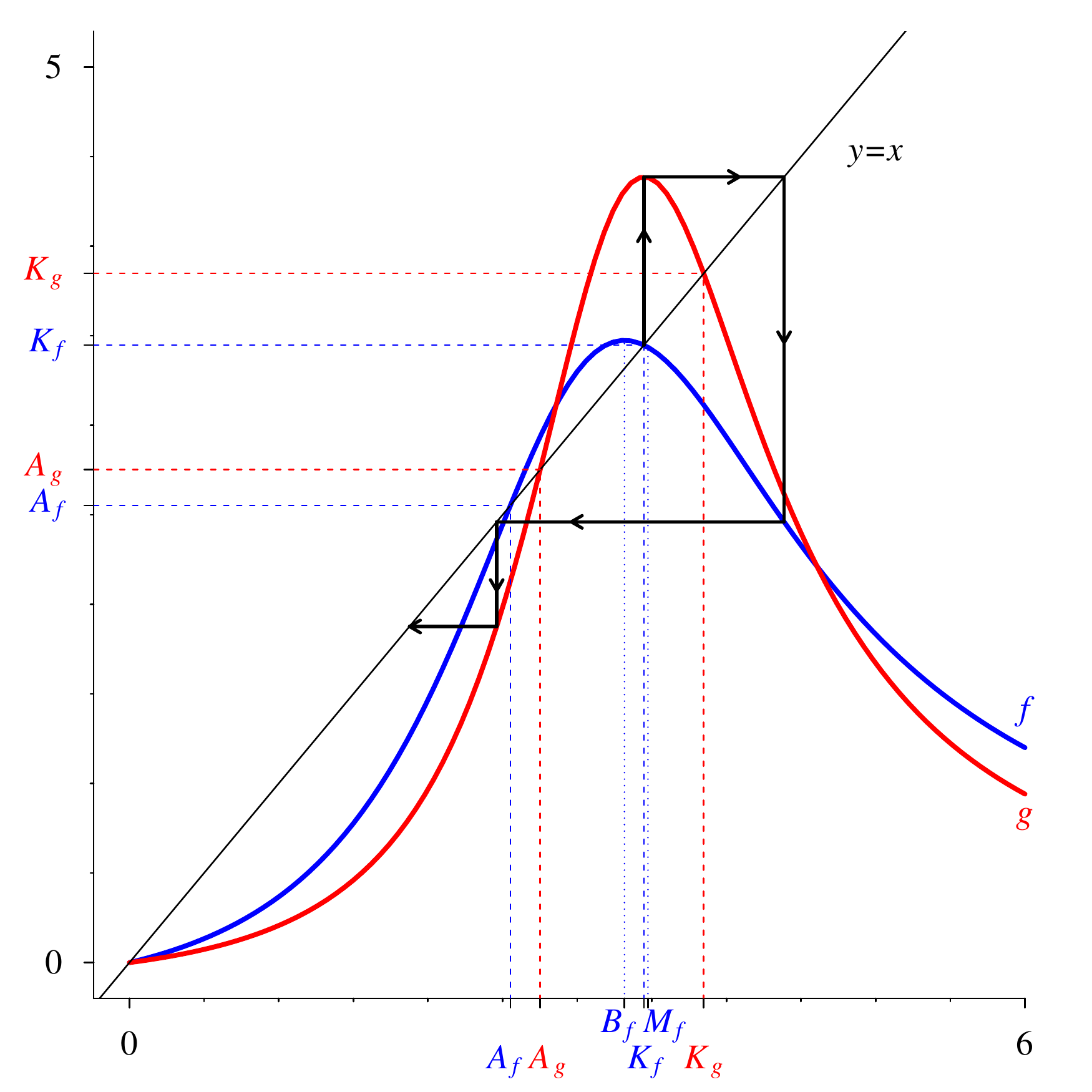}
\end{center}
\caption{Example of functions satisfying the conditions of  Theorem~\ref{th3}.}
\label{fig4}
\end{figure}
Here 
$A_f=3-\sqrt{0.2}\approx 2.553$, 
$K_f=3+\sqrt{0.2}\approx 3.447$, 
$A_g=3.3-\sqrt{0.3}\approx 2.752$ and 
$K_g=3.3+\sqrt{0.3}\approx 3.848$ (hence $A_f<A_g$).
Next, $g(f(g(K_f)))\approx 1.876 < A_f$, so in this case $h_1=g, h_2=f$ and $h_3=g$. Moreover, it can be shown that the function $f$ is concave on the interval $(B_f,M_f)$, hence its first derivative is decreasing on this interval. Since $f^{'}(B_f)=0$ and $f^{'}(M_f)\approx -0.475>-1$, the condition $|f^{'}(x)|<1 $ is satisfied for every $x\in (B_f,M_f)$. Therefore, for the process $\{X_n\}_{n=0}^\infty$ generated by these two functions, we have $P(\lim_{n \to \infty} X_n=0)=1$ for every $x_0\in [0,b]$.
\begin{proof}
Since $M_f$ is the maximum of the function $f$, we obviously have
\begin{equation}
\label{eq7}
P(X_{n+1}\in [0,M_f] \mid X_n\in [0,M])\ge p>0.
\end{equation}
Next, $g(A_f)<A_f$ (because $A_f<A_g$) and since $g$ is continuous, there is some $C \in (A_f,B_f)$ such that $g(x)<A_f$ for every $x \in [A_f,C)$. Hence
\begin{equation}
\label{eq8}
P(X_{n+1}\in [0,A_f) \mid X_n\in [0,C))\ge 1-p > 0.
\end{equation}
Denote $h(x)$ the function defined by $h(x)=h_1\circ h_2\circ \ldots \circ h_m(x)$. This function is continuous (because $f$ and $g$ are continuous), hence there is an $\eta>0$ such that $h(x)<A_f$ for any $x\in (K_f-\eta,K_f+\eta)$. Consequently
\begin{equation}
\label{eq9}
P(X_{n+m}<A_f \mid X_n \in (K_f-\eta,K_f+\eta))>(p(1-p))^m>0.
\end{equation}
Moreover, $ \lim_{n\to \infty} f^n(x)=K_f$ for any $x \in (A_f,M_f]$ (since $|f^{'}(x)|<1$ for every $x \in (B_f,M_f)$), and hence there is an $r\in N$ such that $f^r(x)\in(K_f-\eta,K_f+\eta)$ for any $x\in [C,M_f]$, and so 
\begin{equation}
\label{eq10}
P(X_{n+r} \in (K_f-\eta,K_f+\eta) \mid X_n \in [C,M_f])>p^r>0.
\end{equation}
Now we can gradually apply Corollary~\ref{cor2} to the sets
\begin{itemlist}[(4)]
\item  $[0,M]$ and $[0,M_f]$ (by (\ref{eq7})),
\item  $[0,M_f]$ and $[0,A_f)\cup[C,M_f]$ (if $X_n\in [0,C)$, we can use (\ref{eq8}) and if $X_n\in [C,M_f]$, then $f(X_n)\in [C,M_f]$),
\item  $[0,A_f)\cup [C,M_f]$ and $[0,A_f)\cup (K_f-\eta,K_f+\eta)$ (by (\ref{eq10}) and by the fact that $f([0,A_f))\subseteq [0,A_f)$),
\item  $[0,A_f)\cup (K_f-\eta,K_f+\eta)$ and $[0,A_f)$ (by (\ref{eq9})).
\end{itemlist}
This means that the process $\{X_n\}_{n=0}^\infty$ returns to $[0,A_f)$ infinitely many times, and hence, from previous results (in Section~\ref{sec2}), we have $P(\lim_{n\to \infty} X_n=0)=1$.
\end{proof}
\begin{theorem}
\label{th4}
Let $A_g<A_f$, $f$ be differentiable and $|f^{'}(x)|<1$ for every $x \in (B_f,M_f)$. Assume there exist functions $h_1, h_2,\ldots, h_m \in \{f,g\}$ for which $h_1\circ h_2\circ\ldots\circ h_m(K_f)<A_f$. If, in addition, $f(g(A_f))\neq A_f$, then for every $x_0$ $P(\lim_{n\to\infty} X_n=0)=1$.
\end{theorem}
\begin{proof}
Since $h(K_f)=h_1\circ\ldots\circ h_m(K_f)<A_f$, it follows that $h(K_f)<D$ for some $D<A_f$. As $h$ is continuous, we have 
\begin{equation}
\label{eq11}
P(X_{n+m}<D\mid X_n\in (K_f-\eta,K_f+\eta))>(p(1-p))^m>0 
\end{equation}
for some $\eta>0$.
Since $f(g(A_f))\neq A_f$, two cases are possible: either $f(g(A_f))<A_f$ or $f(g(A_f))>A_f$.
\\ \\
If $f(g(A_f))<A_f$, then $f(g(A_f))<C$ for some $C<A_f$ and from the continuity of $f\circ g$, we have $f(g(x))<C$ for any $x \in (A_f-\xi,A_f+\xi)$, where $0<\xi<\min(A_f-C,A_f-D)$. 
\begin{figure}[h]
\begin{center}
\includegraphics{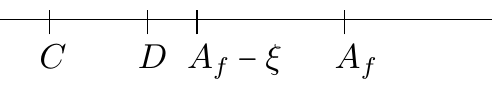}
\end{center}
\caption{Example of possible positions of the points $C, D, A_f-\xi$ and $A_f$.}
\label{fig5}
\end{figure}
Consequently 
\begin{equation}
\label{eq12}
P(X_{n+2}\in [0,A_f-\xi] \mid X_n \in [0,A_f+\xi))\ge p(1-p)>0.
\end{equation}
Again, we can apply Corollary~\ref{cor2} to the sets
\begin{itemlist}[(5)]
\item  $[0,M]$ and $[0,M_f]$ (by (\ref{eq7})),
\item  $[0,M_f]$ and $[0,A_f-\xi]\cup[A_f+\xi,M_f]$ (by (\ref{eq12}) if $X_n \in [0,A_f)$ and by the fact, that $f([A_f+\xi,M_f])\subseteq [A_f+\xi,M_f]$ if $X_n\in[A_f+\xi,M_f]$),
\item  $[0,A_f-\xi]\cup[A_f+\xi,M_f]$ and $[0,A_f-\xi]\cup(K_f-\eta,K_f+\eta)$ (as in the preceding theorem),
\item  $[0,A_f-\xi]\cup(K_f-\eta,K_f+\eta)$ and $[0,A_f-\xi]$ (by (\ref{eq11}), because $D<A_f-\xi$),
\item  $[0,A_f-\xi]$ and $[0,A_g)$ (since $f^n(A_f-\xi)\searrow 0$, and so for some $s\in N$ $f^s(x)<A_g$ for an arbitrary $x\in [0,A_f-\xi]$).
\end{itemlist}
It follows that the process $\{X_n\}_{n=0}^\infty$ returns to the set $[0,A_g)$ infinitely many times, hence $P(\lim_{n\to \infty} X_n=0)=1$. 
\\ \\
If $f(g(A_f))>A_f$, then $f(g(A_f))>C$ for some $C>A_f$ and also $f(g(A_f))\le M_f$ (since $M_f$ is the maximum of the function $f$). Therefore
\begin{equation}
\label{eq13}
P(X_{n+2}\in [A_f+\xi,M_f] \mid X_n \in (A_f-\xi,A_f+\xi))\ge p(1-p)>0, 
\end{equation}
where $0<\xi<\min(A_f-D,C-A_f)$. The rest of the proof is similar to the proof of the first case.
\end{proof}
\textbf{Example.} 
Let $f(x)=\frac{1.155x}{(x-2.8)^2+1.05}$ and $g(x)=\frac{1.3x}{(x-2.9)^2+1}$ (see Fig.~\ref{fig6}).
\begin{figure}[h!]
\begin{center}
\includegraphics[height=8cm,width=8cm]{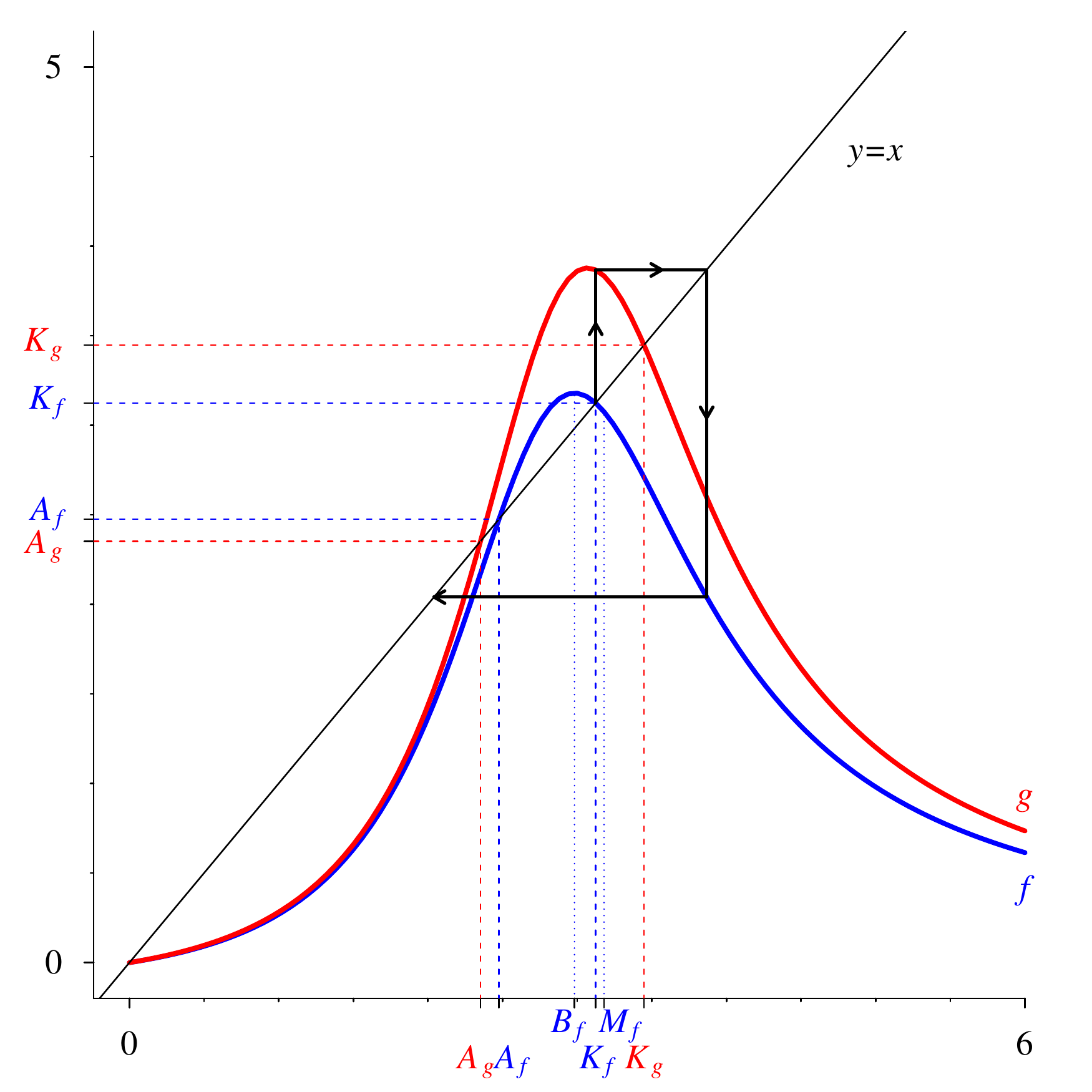}
\end{center}
\caption{Example of functions satisfying the conditions of  Theorem~\ref{th4}.}
\label{fig6}
\end{figure}
Here 
$A_f\approx 2.476,$ 
$K_f\approx 3.124,$ 
$A_g\approx 2.352$ and 
$K_g\approx 3.448$ (hence $A_g<A_f$). 
Further, $f(g(A_f))\approx 2.986\neq A_f$ and $f(g(K_f))\approx 2.041<A_f$. Moreover, the function $f$ is again concave on $(B_f,M_f)$ and $f^{'}(M_f)\approx -0.989 > -1$, hence $|f^{'}(x)|<1$ for every $x\in (B_f,M_f)$. Thus, from Theorem~\ref{th4} it follows that $P(\lim_{n \to \infty} X_n=0)=1$. 

Since $f^n(x)\in [A_f,M_f]$ for every $n\in N$ and for an arbitrary $x_0\in (A_f,M_f)$, the population persists in the deterministic model (model~(\ref{eq1})) generated by the function $f$; similarly the population persists in the model generated by $g$. However, if we combine these models, the population becomes almost surely extinct. This resembles results known as the Parrondo's paradox (see e.g. \cite{Parrondo}).
\\

 A very similar theorem holds for the process $\{Y_n\}_{n=0}^\infty$.
\begin{theorem}
\label{th5}
Let $f$ be a differentiable function and $|f^{'}(x)|<1$ for an arbitrary $x \in (B_f,M_f)$. Assume there exist functions $h_1, h_2,\ldots, h_m \in \{f,g\}$ for which $h_1\circ h_2\circ\ldots\circ h_m(K_f)<A_f$. If, in addition, the set 
\begin{center}
$U=\{x \in [0,\min(A_f,A_g)]: \min(x-f(x),x-g(x))\ge\delta\}$
\end{center} 
is nonempty (see Fig.~\ref{fig7}), then $P(\exists n_0 \in N: Y_n\in[0,\inf U) \ \forall n\ge n_0)=1$ for every $x_0$.
\end{theorem}
\begin{figure}[h]
\begin{center}
\includegraphics[height=8cm,width=8cm]{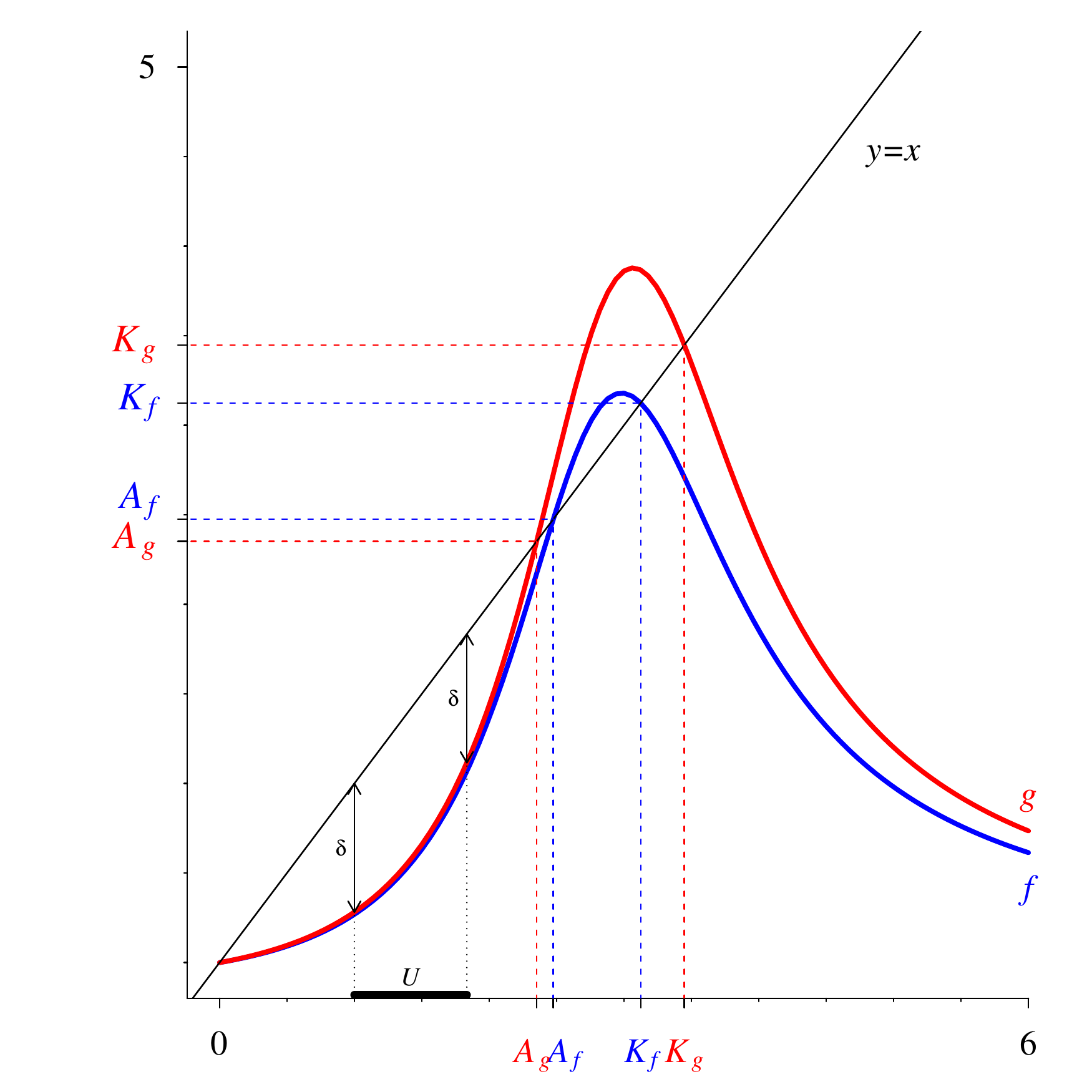}
\end{center}
\caption{Example of functions satisfying the conditions of  Theorem~\ref{th5}.}
\label{fig7}
\end{figure}

The proof is almost the same as the proof of Theorem~\ref{th3} and Theorem~\ref{th4} - inequality (\ref{eq7}) can be modified to 
\begin{equation}
P(Y_{n+1}\in [0,M_f] \mid Y_n\in [0,M])\ge pP(\varepsilon_n<0)>0, 
\end{equation} 
similarly for inequality (\ref{eq8}). Inequality (\ref{eq9}) can be shown as follows: \\
Suppose that $h_1(h_2(K_f))<A_f$ for some $h_1, h_2\in\{f,g\}$ (the proof is the same for $m$ functions). Since $h_1$ is continuous, we have $h_1(x)<A_f$ for an arbitrary $x\in I_1=(h_2(K_f)-\eta_1,h_2(K_f)+\eta_1)$ for some $\eta_1>0$. Next, from the continuity of $h_2$, there exists an $\eta_2>0$ such that $h_2(x)\in I_1$ for every $x\in I_2=(K_f-\eta_2,K_f+\eta_2)$. 
Now let $Y_n\in I_2$. Then $h_2(Y_n)\in I_1$ and since $I_1$ is an open set, there exists a $\mu \in (0,\delta)$ such that $(h_2(Y_n)-\mu, h_2(Y_n)+\mu)\subset I_1$. It follows that 
\begin{equation}
P(Y_{n+1}\in I_1 \mid Y_n \in I_2) \ge p(1-p)P(\varepsilon_n \in (-\mu,\mu))>0.
\end{equation}
Next, if $Y_{n+1}\in I_1$, then $h_1(Y_{n+1})<A_f$, and so there exists a $\nu\in(0,\delta)$ such that $h_1(Y_{n+1})+\nu<A_f$. Hence we obtain
\begin{equation}
P(Y_{n+2}<A_f \mid Y_n\in I_2)\ge (p(1-p))^2P(\varepsilon_n\in (-\mu,\mu))P(\varepsilon_{n+1}\in (-\nu,\nu))>0.
\end{equation}
Inequalities (\ref{eq10}), (\ref{eq11}), (\ref{eq12}) and (\ref{eq13}) can be shown similarly - from the continuity of the functions $f$ and $g$ and by the fact that the density of $\varepsilon_n$ is positive on $(-\delta,\delta)$.
\\ \\
\textbf{Remark.} The condition $f(g(A_f))\neq A_f$ in Theorem~\ref{th4} is not necessary in this case - Corollary~\ref{cor2} can be applied directly to the sets $[0,M_f]$ and $[0,A_f-\frac{\delta}{2}]\cup [A_f+\frac{\delta}{2},M_f]$.

\section{Concluding remarks}
\begin{arabiclist}[(5)]
\item In  Section~\ref{sec2}, we showed that models (\ref{eq2}) and (\ref{eq3}) lead to similar results as the deterministic model, if we work with two strictly increasing Allee maps (recall that we distinguish only between extinction and survival here).
If the initial population size is greater than some threshold, the population persists; if it is smaller than another threshold, the population becomes extinct. The only difference from the non-stochastic case is the region between these two thresholds - if the initial population size is in this region, both scenarios (extinction or survival) are possible.
\item In Theorems~\ref{th1} and ~\ref{th2} only the inequality $A_f<A_g<K_f<K_g$ was considered. If the order of these values is different, proofs and results are very similar. The only difference is in the case $A_f<K_f<A_g<K_g$ (or $A_g<K_g<A_f<K_f$). For these two orderings of the values, it can be easily shown that the population becomes almost surely extinct.
\item If $\delta$ does not satisfy the conditions of Theorem~\ref{th2} or Theorem~\ref{th5} (i.e. $\varepsilon$ can take relatively high values), then it is not possible to predict the long-term behavior of the system. Its values can increase even from zero to high numbers and vice versa.
\begin{figure}[h]
\begin{center}
\includegraphics[height=8cm,width=8cm]{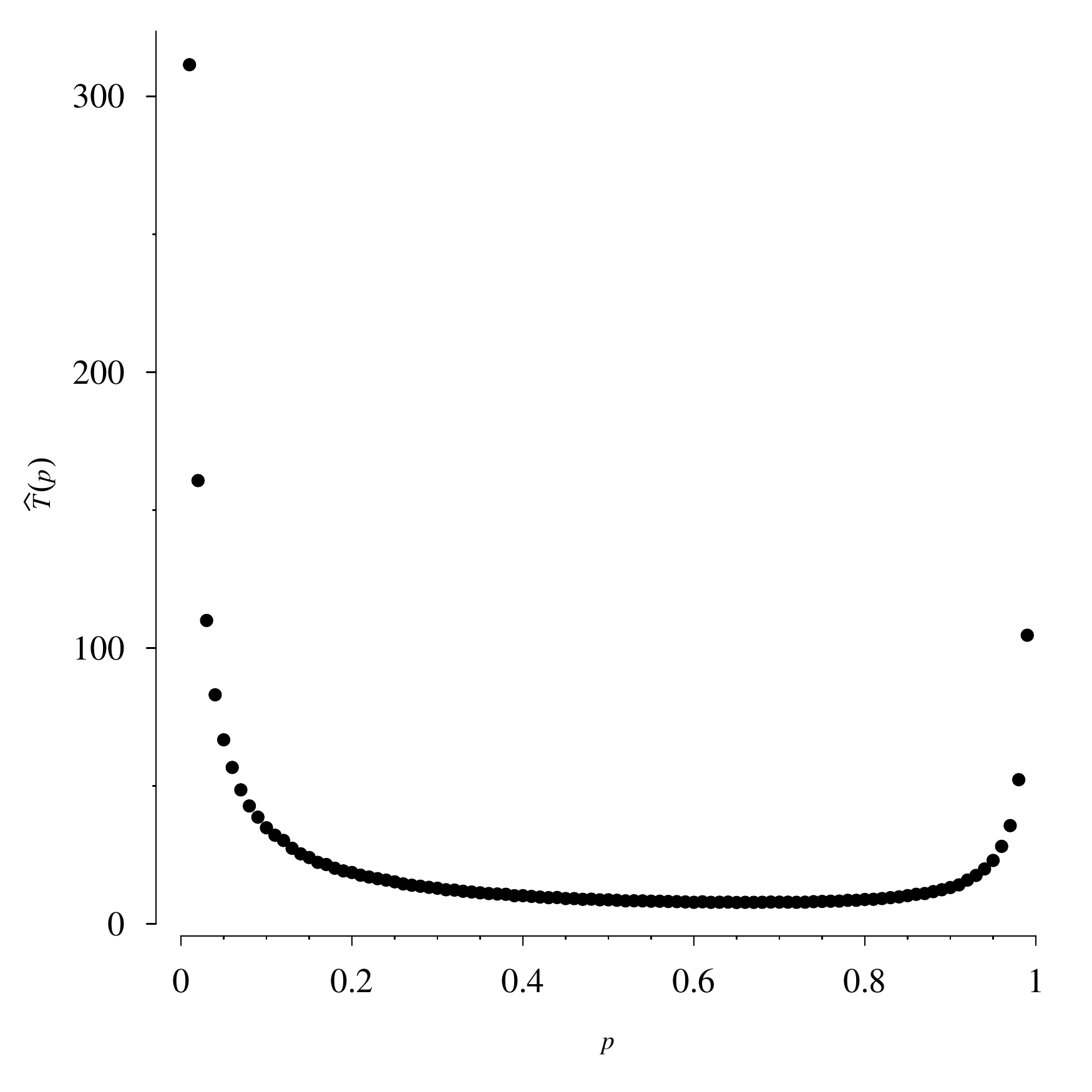}
\end{center}
\caption{Estimates for the expected time $T(p)$.}
\label{fig8}
\end{figure}
\item In Section~\ref{Sec3}, we showed that under some conditions the population becomes almost surely extinct. This is caused by the fact that the population size almost surely drops below the critical value $\min(A_f,A_g)$, which leads to the extinction. Therefore, it would be interesting to calculate the expected value of the time after which the system drops below this critical value. Generally, this can be a very difficult problem, because this time depends on the given functions $f$ and $g$ and also on the probability $p$ with which we choose functions $f$ and $g$. However, if we fix the functions and the probability, the expected time can be estimated by simulations. We examined the expected time $T(p)$ in the case where the functions were the same as in the second example for different probabilities and for $x_0\in (A_f, M_f)$. As previously mentioned (in the second example), if we consider only one function ($f$ or $g$), the population size never drops below the critical value, hence $T(p)\to\infty$ for $p\to 0$ and $p\to 1$. For $p\in (0,1)$ the estimates are shown in Fig.~\ref{fig8}. 
A question is whether it is possible to calculate these values exactly and in more general situations.
\item In  Section~\ref{Sec3}, we considered only specific conditions for unimodal Allee maps, which have interesting consequences. It would be interesting to study other situations.
\end{arabiclist}

\nonumsection{Acknowledgments} \noindent 
Supported by the VEGA Grant No. 2/0047/15 from the Slovak Scientific Grant Agency (Kov\'a\v c, Jankov\'a) and the UK/344/2016 grant from Comenius University in Bratislava (Kov\'a\v c).

\end{document}